\documentclass[a4paper,reqno]{amsart}

\usepackage{amssymb}
\usepackage{amstext}
\usepackage{amsmath}
\usepackage{amscd}
\usepackage{amsthm}
\usepackage{amsfonts}
\usepackage{enumerate}
\usepackage{graphicx}
\usepackage{latexsym}
\usepackage{mathrsfs}
\usepackage{mathtools}
\usepackage[all]{xy}
\xyoption{all}

\usepackage{pstricks}
\usepackage{lscape}
\usepackage{comment}

\newtheorem{theorem}{Theorem}[section]

\newtheorem{corollary}[theorem]{Corollary}
\newtheorem{lemma}[theorem]{Lemma}
\newtheorem{proposition}[theorem]{Proposition}
\newtheorem{definition-proposition}[theorem]{Definition-Proposition}

\theoremstyle{definition}
\newtheorem{definition}[theorem]{Definition}

\newtheorem{example}[theorem]{Example}

\newcommand{\Ext}{\operatorname{Ext}\nolimits}

\newcommand{\Hom}{\operatorname{Hom}\nolimits}
\newcommand{\End}{\operatorname{End}\nolimits}

\renewcommand{\mod}{\mathsf{mod}\hspace{.01in}}

\newcommand{\add}{\mathsf{add}\hspace{.01in}}

\newcommand{\RHom}{\mathbf{R}\strut\kern-.2em\operatorname{Hom}\nolimits}

\numberwithin{equation}{section}

\usepackage{paralist}

\def\Im{\mathop{\rm Im}\nolimits}
\def\Ker{\mathop{\rm Ker}\nolimits}

\def\id{\mathop{\rm id}\nolimits}
\def\pd{\mathop{\rm pd}\nolimits}

\begin{document}
\title{Three results for $\tau$-rigid modules}
\thanks{The authors are supported by NSFC(Nos.11571164 and 11671174) and NSF for Jiangsu
Province (BK20130983)). }

\author{Zongzhen Xie}
\address{Z. Xie: Department of Mathematics, Nanjing University,
Nanjing 210093, P. R. China}
\email{xiezongzhen3@163.com}
\author{Libo Zan}
\address{L. Zan: School of Mathematics and Statistics, NUIST,
Nanjing, 210044, P. R. China} \email{zanlibo@nuist.edu.cn}
\author{Xiaojin Zhang}
\address{X. Zhang: School of Mathematics and Statistics, NUIST,
Nanjing, 210044, P. R. China} \email{xjzhang@nuist.edu.cn}

\maketitle
\begin{abstract}
 $\tau$-rigid modules are essential in the $\tau$-tilting theory introduced by Adachi, Iyama and Reiten.
In this paper, we give equivalent conditions for Iwanaga-Gorenstein
algebras with self-injective dimension at most one in terms of
$\tau$-rigid modules. We show that every indecomposable module over
iterated tilted algebras of Dynkin type is $\tau$-rigid. Finally, we
give a $\tau$-tilting theorem on homological dimension which
is an analog to that of classical tilting modules.
\end{abstract}

\section{Introduction}

In 2014, T. Adachi, O. Iyama and I. Reiten \cite{AIR} introduced
$\tau$-tilting theory to generalize the classical tilting theory.
$\tau$-tilting theory is closely related to silting theory \cite{AI,BZ} and cluster
tilting theory \cite{KR, IY, BMRRT} which are popular in the recent
years. Therefore, $\tau$-tilting theory has attracted widespread
attention. For the latest general results on $\tau$-tilting theory,
we refer to \cite{DIJ, DIRRT, EJR, IJY, IZ1,IZ2, J,K,W} and
references there.

 Note that $\tau$-rigid modules are important
 objects and tools in the $\tau$-tilting theory. It is interesting to study the
 properties of $\tau$-rigid modules and find the indecomposable
 $\tau$-rigid modules for a given algebra. For the recent development of this topic, we refer to \cite{A1,A2,DIP,HZ,Mi,Z1,Z2,Zi1,Zi2} and
 so on. In this paper, we
 also focus on the properties of $\tau$-rigid modules.

For an algebra $A$, denote by $\mathrm{mod}A$ the category of finitely generated right $A$-modules. Recall that an algebra $A$ is Iwanaga-Gorenstein, that is, $\id_{A}
A< \infty$ and $\id_{A^{\rm op}} A< \infty$. In this case, $\id_{A}
A=\id_{A^{\rm op}} A$. Our first main result gives some
new equivalent conditions for an Iwanaga-Gorenstein algebra $A$ with
$\id_A A\leq 1$ in terms of $\tau$-rigid modules. We remark that
this result was inspired by Osamu Iyama and Yingying Zhang.

\begin{theorem}\label{0.1}(Theorems \ref{2.2},\ref{2.4}) For an algebra $A$, the following are equivalent
 \begin{enumerate}[\rm(1)]
 \item $A$ is Iwanaga-Gorenstein with $\mathrm{id}_{A}A\leq 1$.
 \item Every classical cotilting module in $\mathrm{mod}A$ is a classical tilting module.
\item $\mathbb{D}A$ is a $\tau$-rigid module in $\mathrm{mod}A$.
\item $A$ is a $\tau^{-1}$-rigid module in $\mathrm{mod}A$.
\end{enumerate}
\end{theorem}

We are also interested in algebras $A$ satisfying every indecomposable module in $\mathrm{mod}A$ is $\tau$-rigid.
Easy examples of such algebras are hereditary algebras of Dynkin type. We aim to find more examples in this paper.
Recall that  Assem and Happel \cite{AsH} introduced the following notation of iterated tilted algebras of Dynkin type as a generalization of tilted algebras of Dynkin type \cite{HR}. Let $Q$ be a
finite, connected, and acyclic quiver. An algebra $A_m$ ($m\geq1$) is called an iterated
tilted algebra of type $Q$ if (1) $A_0=KQ$, (2) $T_i$ is a splitting
classical tilting module in $\mathrm{mod}A_i$ and (3) $A_{i+1}=\End_{A_i}T_i$ are satisfied, where $0\leq i \leq m-1$. Our second main result is the following.

\begin{theorem}\label{0.4}(Theorem \ref{4.7})
Let $B$ be an iterated tilted algebra of Dynkin type. Then every
indecomposable module in $\mathrm{mod}B$ is $\tau$-rigid.
\end{theorem}
For a classical tilting module $T$ in $\mathrm{mod} A$ with $B=\End_A T$,
by using the tilting theorem of Brenner and Butler \cite{BB}, one gets that the homological dimension of
$N\in \mathrm{Fac} T$ gives an upper bound of the homological dimension of
$\Hom_{A}(T,N)$, where $\mathrm{Fac} T$ (resp. $ \mathrm{Sub} T$) is the
subcategory of $\mathrm{mod}A$ consisting of modules $N$ generated (resp. cogenerated)
by $T$. It is natural to ask: Is there a similar
result for $\tau$-tilting modules? We give a positive answer to this
question and get our third main result. We should remark that Buan
and Zhou have studied the global dimension of $2$-term silting
complexes in [BZ].

\begin{theorem}\label{0.3}(Theorems \ref{3.2}, \ref{3.a}) Let $A$ be an algebra, $T$ be a $\tau$-tilting module in $\mod A$,
$B=\mathrm{End}_{A} T$ and $C=\mathrm{End}_{A} \tau T^{\rm op}$.
 \begin{enumerate}[\rm(1)]
\item    For any $M\in \mathrm{Fac}T$ with $\pd_{A} M\leq 1$, $\mathrm{pd}_{B}\mathrm{Hom}_{A}(T,M)\leq \mathrm{pd}_{A}M$
holds.
\item    For any $M\in \mathrm{Fac}T$ with $\Ext_{A}^{i}(T,M\oplus T)=0$
for any $i\geq 1$, $\mathrm{pd}_{B}\mathrm{Hom}_{A}(T, M)\leq
\mathrm{pd}_{A}M$ holds.
\item  For any $N\in \mathrm{Sub}\tau T$ with $\id_{A} N\leq 1$, $\mathrm{pd}_{C}
\mathrm{Hom}_{A}(N,\tau T)\leq \mathrm{id}_{A}N$ holds.
\item   For any $N\in \mathrm{Sub}\tau T$ with $\Ext_{A}^{i}(N\oplus \tau T,\tau T)=0$
for any $i\geq 1$, $\mathrm{pd}_{C}\mathrm{Hom}_{A}(N,\tau T)\leq
\mathrm{id}_{A}N$ holds.

\end{enumerate}
\end{theorem}

The paper is organized as follows:

In Section 2, we study $\tau$-rigid modules over
Iwanaga-Gorenstein algebras and show Theorem \ref{0.1}. In Section
3, we study the indecomposable $\tau$-rigid modules over iterated tilted
algebras of Dynkin type and show Theorem \ref{0.4}. In Section 4, we give the
$\tau$-rigid (resp. $\tau^{-1}$-rigid) version of Wakamastu's Lemma
and then give an upper bound for some special modules over the
endomorphism ring of $\tau$-tilting (resp. $\tau^{-1}$-tilting)
modules.

Throughout this paper, all algebras are finite dimensional algebras over an algebraically closed field
 $K$ and $\mathbb{D}=\mathrm{Hom}_{K} (-,K)$ is the standard duality.
\section{Gorenstein algebras and $\tau$-rigid modules}
In this section, we aim to study Iwanaga-Gorenstein algebras in
terms of $\tau$-rigid modules.

For an algebra $A$, denote by $\mathrm{gl.dim}A$ the global dimension of $A$. For a right $A$-module $M$, denote by $\mathrm{pd}_A M$ (resp. $\mathrm{id}_A M$) the projective dimension (resp. injective dimension) of $M$, denote by ${\mathrm{add}_A M}$ the subcategory
of direct summands of finite direct sums of $M$ and denote by $|M|$ the number of pairwise
nonisomorphic indecomposable summands of $M$. Firstly, we recall
the definition of tilting (resp. cotilting) modules, see \cite{M}
for details.

\begin{definition}\label{2.a} A module $T\in\mathrm{mod}A$ is called a {\it tilting} module, if it satisfies
 \begin{enumerate}[\rm(1)]
\item $\mathrm{pd}_A T \leq n $.
\item $\mathrm{Ext} _{A}^{i} (T,T) =0$ for all $i\geq 1$.
\item There exists an exact sequence $0 \rightarrow A \rightarrow T_{0} \rightarrow T_{1}\rightarrow \cdots \rightarrow T_{n} \rightarrow 0$, for all $T_{i} \in  \mathrm{add} T$, $0 \leq i \leq n$.
\end{enumerate}
\end{definition}

In particular, we call $T$ in Definition \ref{2.a} a {\it classical
tilting} module whenever $n=1$. In this case, Definition 2.1(3) is equivalent to $|T|=|A|$. Dually, one can define {\it
cotilting} modules and {\it classical cotilting} modules.

 We also need the following
definitions in \cite{AIR}.

\begin{definition}
 \begin{enumerate}[\rm(1)]
\item We call $T \in \mathrm{mod}A$ {\it $\tau$-rigid} if $\mathrm{Hom}_{A}(T,\tau T)=0$,
where $\tau$ is the Auslander-Reiten translation. Moreover, $T$ is called {\it $\tau$-tilting} if $T$ is $\tau$-rigid
and $|T|=|A|$.
\item We call $T \in \mathrm{mod}A$ {\it $\tau^{-1}$-rigid} if $\mathrm{Hom}_{A}(\tau^{-1}T,T)=0$. Moreover, $T$ is called {\it$\tau^{-1}$-tilting} if $T$ is $\tau^{-1}$-rigid and $|T|=|A|$.
\end{enumerate}
\end{definition}

Clearly, $T$ is $\tau^{-1}$-rigid (resp. $\tau^{-1}$-tilting)
module in $\mathrm{mod}A$ if and only if $\mathbb{D}T$ is $\tau$-rigid (resp.
$\tau$-tilting) module in $\mathrm{mod}A^{op}$.

Recall that $T \in \mathrm{mod}A$ is called {\it faithful} if the
right annihilator of $T$ is zero. Now we can state the following proposition
in \cite{AIR}.

\begin{proposition}\label{2.b}
 \begin{enumerate}[\rm(1)]
\item Any faithful $\tau$-rigid module $T$ in $\mathrm{mod}A$ is a partial tilting $A$-module, that is, $\mathrm{Ext}_{A}^{1}(T,T)=0$ and $\mathrm{pd}_{A}T \leq1$.
\item Any faithful $\tau$-tilting module in $\mathrm{mod}A$ is a classical tilting $A$-module.
\end{enumerate}
\end{proposition}

Now we can state the properties of $\tau$-rigid cotilting modules,
which is essential in the proof of the main result.

\begin{proposition}\label{2.1}
 \begin{enumerate}[\rm(1)]
\item If a cotilting (resp. tilting) module $T$ in $\mathrm{mod}A$ is $\tau$-rigid, then $T$ is a classical tilting $A$-module.
\item If a tilting (resp. cotilting) module $T$ in $\mathrm{mod}A$ is $\tau^{-1}$-rigid, then $T$ is a classical cotilting $A$-module.
\end{enumerate}

\end{proposition}

\begin{proof}
We only prove (1), since the proof of (2) is similar. Because $T$ is
cotilting, $T$ is faithful by \cite[Chapter VI, Lemma
2.2]{AsSS}. By Proposition \ref{2.b}, any faithful
$\tau$-rigid module in $\mathrm{mod}A$ is a partial tilting $A$-module. Note that
$|T|=|A|$, we are done.
\end{proof}

For any $X\in \mathrm{mod} A$, denote by $\mathrm{Fac}X =\{M| X^{n}
\twoheadrightarrow M\ for\ some\ integer\ n\}$. The
following proposition \cite[Proposition 5.8]{AS} are useful.

\begin{proposition}For $X$ and $Y$ in $\mathrm{mod} A$, we have the following
 \begin{enumerate}[\rm(1)]
\item  $\mathrm{Hom}_{A}(X,\tau Y)=0$ if and only if
$\mathrm{Ext}_{A}^{1}(Y,\mathrm{Fac}X)=0$.
\item $X$ is $\tau$-rigid if and only if $\mathrm{Ext}_{A}^{1}(X,\mathrm{Fac}X)=0$.
\end{enumerate}
\end{proposition}

For an algebra $A$, denote by $\mathcal{P} (A)$ the subcategory of finitely generated projective right $A$-modules
 and denote by $\mathcal{I}(A)$ the subcategory of finitely generated injective right $A$-modules.
 Now we recall the following result due to Happel and Unger \cite[Lemma
1.3]{HU}. We provide a new proof for this result.

\begin{theorem}\label{2.2} For an algebra $A$, the following are equivalent.
 \begin{enumerate}[\rm(1)]
 \item $A$ is Iwanaga-Gorenstein.
 \item Every cotilting module in $\mathrm{mod} A$ is tilting.
 \item Every tilting module in $\mathrm{mod} A$ is cotilting.
\item There exists a tilting-cotilting module in $\mathrm{mod} A$.

\end{enumerate}
\end{theorem}

\begin{proof}

$(1)\Rightarrow (2)$  Assume that $A$ is Iwanaga-Gorenstein and $T$
is a cotilting module in $\mathrm{mod}A$. Since $T$ is
selforthogonal and $\mathrm{id}_{A} T$ is finite, we only need to show
that every module in $\mathcal{P} (A)$ has a finite exact
coresolution in $\mathrm{add} T$.

   Denote by
    $^{\bot} T =\{M\in \mathrm{mod} A \mid \mathrm{Ext}_{A}^{i} (M,T)=0\ for\ i\geq 1\}$ and $$ X_{T}=\{X\mid 0\rightarrow X \rightarrow T_{0} \stackrel{f_0}{\rightarrow} T_{1} \rightarrow \cdots \rightarrow T_{n} \stackrel{f_n}{\rightarrow} T_{n+1} \rightarrow \cdots, T_{i} \in \mathrm{add} T, \mathrm{Im} f_{n} \in ^{\bot} T, n\geq 0\}.$$ Since $T$ is a cotilting $A$-module, we have that $\mathcal{P} (A)$ is contained in $\mathcal{X}_{T} ={^{\bot}T}$.

    Then there exists an exact sequence
    $$0\rightarrow P \rightarrow T_{0} \stackrel{f_0}{\rightarrow} T_{1} \rightarrow \cdots \rightarrow T_{n} \stackrel{f_n}{\rightarrow} T_{n+1} \rightarrow \cdots$$
    for all $A$-module $P\in\mathcal{P} (A)$, where $T_{i}\in\mathrm{add} T$ and $X_{i}=\mathrm{Im} f_{i}$ is in $\mathcal{X} _{T}$ for all $i\geq 0$. Let $\mathrm{id}_{A} \mathcal{P} (A) \leq r$. Then
    $\mathrm{Ext}_{A}^{1} (X_{r},X_{r-1}) =\mathrm{Ext}_{A}^{2} (X_{r},X_{r-2})=\cdots =\mathrm{Ext}_{A}^{r+1} (X_{r},P) =0,$
    hence the exact sequence $0\rightarrow X_{r-1} \rightarrow T_{r} \rightarrow X_{r} \rightarrow 0$
     splits, such that $X_{r-1}\in\mathrm{add} T$. This implies that $T$ is a tilting $A$-module.

$(2)\Rightarrow (1)$  Assume that a module $T$ in $\mathrm{mod}A$ is a
cotilting-tilting module. By the definitions of cotilting and
tilting modules every module in $\mathcal{I} (A)$ has a finite exact
resolution in $\mathrm{add} T$ and every module in $\mathcal{P} (A)$
has a finite exact coresolution in $\mathrm{add} T$. Since
$\mathrm{id}_{A} T$ and $\mathrm{pd}_{A} T$ both are finite, it
follows immediately that both $\mathrm{pd}_{A}  \mathcal{I} (A)$ and
$\mathrm{id}_{A} \mathcal{P} (A)$ are finite. Therefore $A$ is
Gorenstein.

Similarly, one can prove the equivalence of (1) and (3).

In the following we show the equivalence of (1) and (4).

 $(1)\Rightarrow (4)$  Assume that $A$ is
Gorenstein. Then $A$ is a tilting-cotilting
module.

 $(4)\Rightarrow (1)$ is similar to
$(2)\Rightarrow (1)$.
\end{proof}

Now we are in a position to show the main result in this section.
\begin{theorem}\label{2.4} For an algebra $A$, the following are equivalent.
 \begin{enumerate}[\rm(1)]
\item $A$ is Iwanaga-Gorenstein with $\mathrm{id}_{A}A\leq 1$.
\item $\mathbb{D}A$ is a $\tau$-rigid module in $\mathrm{mod}A$.
\item $A$ is a $\tau^{-1}$-rigid module in $\mathrm{mod}A$.
\end{enumerate}
\end{theorem}

\begin{proof}
We show the equivalence of (1) and (2). Similarly, one can show the
equivalence of (1) and (3).

$(1)\Rightarrow(2)$ For any $M\in \mathrm{Fac}\mathbb{D}A$, there
exists a short exact sequence
$$0\rightarrow N \rightarrow \mathbb{D}A^{n} \rightarrow M \rightarrow 0  \qquad (2.1) $$
Applying the functor $\mathrm{Hom}_{A}(\mathbb{D}A,-)$ to the short
exact sequence $(2.1)$ yields the following long exact sequence
$$0 \rightarrow \mathrm{Hom}_{A}(\mathbb{D}A,N) \rightarrow \mathrm{Hom}_{A}(\mathbb{D}A,\mathbb{D}A^{n})
 \rightarrow \mathrm{Hom}_{A}(\mathbb{D}A,M) \rightarrow \mathrm{Ext}_{A}^{1}(\mathbb{D}A,N)$$
 $$\rightarrow \mathrm{Ext}_{A}^{1}(\mathbb{D}A,\mathbb{D}A^{n}) \rightarrow \mathrm{Ext}_{A}^{1}(\mathbb{D}A,M)
 \rightarrow \mathrm{Ext}_{A}^{2}(\mathbb{D}A,N) \rightarrow \mathrm{Ext}_{A}^{2}(\mathbb{D}A,\mathbb{D}A^{n}) \rightarrow
  \cdots$$
Then $\mathrm{Ext}_{A}^{1}(\mathbb{D}A,M) \simeq
\mathrm{Ext}_{A}^{2}(\mathbb{D}A,N)$ since $\mathbb{D}A$ is an
injective $A$-module, and $\mathrm{pd}_{A}\mathbb{D}A \leq 1$ since
$\mathrm{id}_{A}A \leq 1$. Thus $\mathrm{Ext}_{A}^{1}(\mathbb{D}A,M)
\simeq \mathrm{Ext}_{A}^{2}(\mathbb{D}A,N) =0$. We have
$\mathrm{Ext}_{A}^{1}(\mathbb{D}A,\mathrm{Fac}\mathbb{D}A) =0$,
therefore $\mathbb{D}A$ is a $\tau$-rigid $A$-module by Proposition
2.5. Since $|\mathbb{D}A|=|A|$, one gets $\mathbb{D}A$ is a
$\tau$-tilting $A$-module.

$(2)\Rightarrow(1)$ Since $\mathbb{D}A$ is $\tau$-rigid and
$|\mathbb{D}A|=|A|$, $\mathbb{D}A$ is a $\tau$-tilting
$A$-module. By \cite[ChapterVI,
Lemma 2.2]{AsSS}, $\mathbb{D}A$ is faithful. Then $\mathbb{D}A$ is a classical tilting
$A$-module by Proposition \ref{2.b}(2), and hence
$\mathrm{pd}_{A}\mathbb{D}A \leq 1$. Thus $\mathrm{id}_{A^{\rm
op}}A \leq 1$. Note that $\mathrm{id}_{A}A \leq 1$ if and only if
$\mathrm{id}_{A^{\rm op}}A\leq 1$, then $\mathrm{id}_{A}A \leq 1$.
\end{proof}

The following corollary is immediate.

\begin{corollary}\label{2.5}For an algebra $A$, if one of the following conditions is satisfied,
 \begin{enumerate}[\rm(1)]
\item Every $\tau$-tilting $A$-module is a $\tau^{-1}$-tilting $A$-module;
\item Every $\tau^{-1}$-tilting $A$-module is a $\tau$-tilting $A$-module,
\end{enumerate}
then $A$ is Iwanaga-Gorenstein with $\mathrm{id}_{A}A\leq1$.
\end{corollary}

\begin{proof}
We only prove (1) since the proof of (2) is similar. By assumption
we have that the $\tau$-tilting module $A$ is a $\tau^{-1}$-tilting
module. Then $A$ is a Gorenstein algebra with $\mathrm{id}_{A}A\leq
1$ by Theorem \ref{2.4}(3).
\end{proof}

We should remark that the converse of Corollary 2.8 is not true in
general.

\begin{example}

Let $A$ be the algebra given by the quiver $\xymatrix{
1\ar@<2pt>[r]^{\alpha}&2\ar@<2pt>[l]^{\beta}
}$ with relations $\alpha \beta =\beta \alpha =0$. Then the support $\tau$-tilting quiver of $A$ is the following:\\
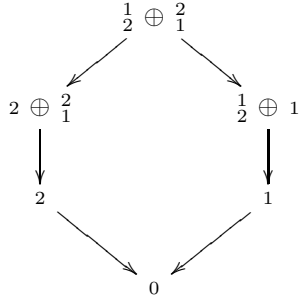
$$\begin{xy}
 (0,0)*+{\begin{smallmatrix}1\\2\end{smallmatrix} \oplus \begin{smallmatrix}2\\1\end{smallmatrix}}="1",
(-15,-12)*+{\begin{smallmatrix}2\end{smallmatrix} \oplus \begin{smallmatrix}2\\1\end{smallmatrix}}="2",
(15,-12)*+{\begin{smallmatrix}1\\2\end{smallmatrix} \oplus \begin{smallmatrix}1\end{smallmatrix}}="3",
(-15,-24)*+{\begin{smallmatrix} 2\\
\end{smallmatrix}}="4",
(15,-24)*+{\begin{smallmatrix} 1\\
\end{smallmatrix}}="5",
(0,-36)*+{\begin{smallmatrix} 0\\
\end{smallmatrix}}="6",
\ar"1";"2", \ar"1";"3", \ar"2";"4", \ar"3";"5", \ar"4";"6",
\ar"5";"6",
\end{xy}$$

One can show that ${\begin{smallmatrix}2\end{smallmatrix} \oplus \begin{smallmatrix}2\\1\end{smallmatrix}}$
 is $\tau$-tilting but not $\tau^{-1}$-tilting.

\end{example}

At the end of this section, we give an example to show that the
existence of $\tau$-tilting-$\tau^{-1}$-tilting modules (even
$\tau$-rigid classical cotilting modules) is not equivalent to
$1$-Gorensteiness in general.
\begin{example}
Let $A$ be the algebra given by the quiver $1\xrightarrow {\alpha} 2
\xrightarrow {\beta} 3 $ with the relation $\alpha\beta=0$. Then
$T={\begin{smallmatrix}1\\2\end{smallmatrix} \oplus \begin{smallmatrix}2\\3\end{smallmatrix} \oplus \begin{smallmatrix}2\end{smallmatrix}}$ is a
$\tau$-tilting-$\tau^{-1}$-tilting module in $\mathrm{mod} A$ (actually a classical
tilting-cotilting module) but $\mathrm{gl.dim}A=2$.
\end{example}

\section{iterated tilted algebras and $\tau$-rigid modules}

In this section, we focus on the $\tau$-rigid modules over iterated tilted
algebras and show every indecomposable module over an iterated
tilted algebra of Dynkin type is $\tau$-rigid. Throughout this section, all tilting modules are classical tilting modules.

Firstly, we need the notion of torsion pairs.

\begin{definition}\label {4.1} Let $A$ be an algebra. A pair $(\mathcal{T}, \mathcal{F})$ of full subcategories of $\mathrm{mod} A$ is called a {\it torsion pair} if the following conditions are satisfied:
 \begin{enumerate}[\rm(1)]
\item $\mathrm{Hom}_{A}(M,N)=0$ for all $M\in \mathcal{T},N\in \mathcal{F}$.
\item $\mathrm{Hom}_{A}(M,-)|_{\mathcal{F}}=0$ implies $M\in \mathcal{T}$.
\item $\mathrm{Hom}_{A}(-,N)|_{\mathcal{T}}=0$ implies $N\in \mathcal{F}$.
\end{enumerate}
\end{definition}

To introduce the tilting theorem due to Brenner and Bulter, we also
need the following:

\begin{definition}\label{4.2}
Let $A$ be an algebra. Any tilting module $T$ in $\mathrm{mod} A$ induces {\it torsion pairs} $(\mathcal{T}(T),\mathcal{F}(T))$ in $\mathrm{mod}A$ and $(\mathcal{X}(T),\mathcal{Y}(T))$
 in $\mathrm{mod}B$ with $B=\mathrm{End_{A}}T$, where $$\mathcal{T}(T)=\{M_{A}|\mathrm{Ext}_{A}^{1}(T,M)=0\},$$ $$\mathcal{F}(T)=\{M_{A}|\mathrm{Hom}_{A}(T,M)=0\},$$
$$\mathcal{X}(T)=\{X_{B}|\mathrm{Hom}_{B}(X,\mathbb{D}T)=0\}=\{X_{B}|X\otimes _{B}T=0\},$$
$$\mathcal{Y}(T)=\{Y_{B}|\mathrm{Ext}_{B}^{1}(Y,\mathbb{D}T)=0\}=\{Y_{B}|\mathrm{Tor}_{1}^{B}(Y,T)=0\}.$$
\end{definition}

Now we can state the tilting theorem of Brenner and Bulter \cite{BB} as follows:

\begin{theorem}\label{4.3}
Let $A$ be an algebra, $T$ be a tilting module in $\mathrm{mod} A$ and
$B=\mathrm{End}_{A}T$. Let $(\mathcal{T}(T),
\mathcal{F}(T))$ and $(\mathcal{X}(T), \mathcal{Y}(T))$ be
the induced torsion pairs in $\mathrm{mod}A$ and $\mathrm{mod}B$,
respectively. Then $T$ has the following properties:
\begin{enumerate}[\rm(1)]
\item $_{B}T$ is a tilting $B$-module, and the canonical $K$-algebra homomorphism $A\rightarrow\mathrm{End}_{B}T^{\mathrm{op}}$ defined by $a\mapsto(t\mapsto ta)$ is an isomorphism.
\item The functors $\mathrm{Hom}_{A}(T,-)$ and $-\otimes_{B}T$ induce quasi-inverse equivalences between $\mathcal{T}(T)$ and $\mathcal{Y}(T)$.
\item The functors $\mathrm{Ext}_{A}^{1}(T,-)$ and $\mathrm{Tor}_{1}^{B}(-,T)$ induce quasi-inverse equivalences between $\mathcal{F}(T)$ and $\mathcal{X}(T)$.
\end{enumerate}
\end{theorem}

Recall that a torsion pair $(\mathcal{T}, \mathcal{F})$ in
$\mathrm{mod}A$ is called {\it splitting} if for any indecomposable
$M\in\mathrm{mod}A$ either $M\in\mathcal{T}$ or $M\in\mathcal{F}$
holds. For a tilting module $T_{A}$ with
$B=\mathrm{End}_{A}T$,  $T_{A}$ is said to be {\it splitting} if the
induced torsion pair $(\mathcal{X}(T), \mathcal{Y}(T))$ in
$\mathrm{mod}B$ is splitting. The following propositions in \cite{AsSS} are critical in the
proof of the main result in this section.

\begin{proposition}\label{4.4}\cite[VI, Corollary 5.7]{AsSS}
For an algebra $A$, if $\mathrm{gl.dim}A\leq 1$, then every tilting module in $\mathrm{mod}A$ is splitting.
\end{proposition}

\begin{proposition}\label{4.5}\cite[VI, Proposition 5.2]{AsSS}
Let $A$ be an algebra, $T$ be a splitting tilting module in $\mathrm{mod}A$,
and $B=\mathrm{End}_{A}T$. Then any almost split sequence in
$\mathrm{mod}B$ lies entirely in either $\mathcal{X}(T)$ or
$\mathcal{Y}(T)$, or else it is of the form
$$0\rightarrow \mathrm{Hom}_{A}(T,I) \rightarrow \mathrm{Hom}_{A}(T,I/\mathrm{soc} I)\oplus \mathrm{Ext}_{A}^{1}(T,\mathrm{rad} P) \rightarrow \mathrm{Ext}_{A}^{1}(T,P) \rightarrow 0,$$
where $P$ is an indecomposable projective $A$-module not lying in $\mathrm{add}T$ and $I$ is the indecomposable injective $A$-module such that $P/\mathrm{rad} P \cong \mathrm{soc} I$.
\end{proposition}

Keeping the symbols as above, we can recall the following
proposition.

\begin{proposition}\label{4.6}\cite[VI, Lemma 5.3]{AsSS}
Let $0\rightarrow L \rightarrow M \rightarrow N \rightarrow 0$ be an
almost split sequence in $\mathrm{mod} B$.
 \begin{enumerate}[\rm(1)]
\item If $L,M,N\in \mathcal{Y}(T)$ , then $0\rightarrow L\otimes _{B}T \rightarrow M\otimes _{B}T \rightarrow N\otimes _{B}T \rightarrow 0$ is almost split in $\mathcal{T}(T)$.
\item If $L,M,N\in \mathcal{X}(T)$, then $0\rightarrow \mathrm{Tor}_{1}^{B}(L,T) \rightarrow \mathrm{Tor}_{1}^{B}(M,T) \rightarrow \mathrm{Tor}_{1}^{B}(N,T) \rightarrow 0$ is almost split in $\mathcal{F}(T)$.

\end{enumerate}
\end{proposition}

Let $Q$ be a finite, connected, and acyclic quiver. Recall that an
algebra $B$ is called an {\it iterated tilted algebra of type $Q$}
if there is a series of algebras $A_0=KQ,A_1,\cdots, A_m=B$ such
that $T_i$ is a splitting classical tilting module over $A_i$ and
$A_{i+1}=\End_{A_i}T_i$ for $0\leq i \leq m-1$. Now we are in a position to show the
main result of this section.

\begin{theorem}\label{4.7}
Let $B$ be an iterated tilted algebra of Dynkin type $Q$. Then every
indecomposable module in $\mathrm{mod}B$ is $\tau$-rigid.
\end{theorem}

\begin{proof}

Assume that $B=A_m$ is the iterated tilted algebra of Dynkin type
$Q$ with the corresponding splitting tilting modules $T_i$ for
$0\leq i\leq m-1$. We prove the assertion by induction on $m$.

If $m=1$, then $B=A_1=\End_{A_0}T_0$ is a tilted algebra of Dynkin
type.

Let $N$ be any indecomposable module in $\mathrm{mod} B$. By
Proposition \ref{4.4}, $N$ is either in $\mathcal{X}(T)$ or in
$\mathcal{Y}(T)$.

If $N$ is projective, then there is nothing to show. Now assume that
$N$ is not projective. Then there is an almost split sequence
$0\rightarrow L\rightarrow M\rightarrow N\rightarrow 0$. By
Propostion \ref{4.5}, the exact sequence is either in
$\mathcal{Y}(T)$, $\mathcal{X}(T)$ or a connecting sequence.

(1) If $0\rightarrow L \rightarrow M \rightarrow N \rightarrow 0$ in $\mathcal{Y}(T)$, then
 $0\rightarrow L\otimes _{B}T \rightarrow M\otimes _{B}T \rightarrow N\otimes _{B}T \rightarrow 0$
 is Auslander-Reiten sequence in $\mathrm{mod}A_{0}$ by Proposition \ref{4.6}. Since $A_{0}$ is the path algebra of a Dynkin quiver, $A_{0}$ is a representation-finite hereditary algebra. This implies every indecomposable module in $\mathrm{mod}A_{0}$ is directing and thus $\tau$-rigid. By Theorem \ref{4.3}, $\mathrm{Hom}_{B}(N,L)\simeq \mathrm{Hom}_{A_{0}}(N\otimes _{B}T,L\otimes _{B}T)=0$, hence $N$ is $\tau$-rigid.

(2) If $0\rightarrow L \rightarrow M \rightarrow N \rightarrow 0$ in
$\mathcal{X}(T)$, then $0\rightarrow \mathrm{Tor}_{1}^{B}(L,T)
\rightarrow \mathrm{Tor}_{1}^{B}(M,T) \rightarrow
\mathrm{Tor}_{1}^{B}(N,T) \rightarrow 0$ is Auslander-Reiten
sequence in $\mathrm{mod}A_{0}$ by Proposition \ref{4.6}. As we showed in (1) every indecomposable module in $\mathrm{mod}A_{0}$ is $\tau$-rigid. By Theorem \ref{4.3}, $\mathrm{Hom}_{B}(N,L)\simeq
\mathrm{Hom}_{A_{0}}(\mathrm{Tor}_{1}^{B}(N,T),\mathrm{Tor}_{1}^{B}(L,T))=0$,
hence $N$ is $\tau$-rigid.

(3) If $0\rightarrow L \rightarrow M \rightarrow N \rightarrow 0$ is
a connecting sequence, then $N\simeq \mathrm{Ext}_{A_{0}}^{1}(T,P(a))\in
\mathcal{X}(T)$, $L\simeq \mathrm{Hom}_{A_{0}}(T,I(a))\in
\mathcal{Y}(T)$ by Proposition \ref{4.5}.  Thus,
$\mathrm{Hom}_{B}(N,\tau N)=\mathrm{Hom}_{B}(N,L)=0$, $N$ is
$\tau$-rigid.

Now assume the assertion holds for $B=A_m$. In the following we show
the assertion holds for $B=A_{m+1}$.

By induction assumption, every indecomposable module in $\mathrm{mod}A_m$ is
$\tau$-rigid. For any indecomposable module $N\in\mathrm{mod} B$, if $N$ is
projective, then there is nothing to show. We assume that $N$
 is not projective. Since $T_m$ is splitting, then $N$ is either in $\mathcal{Y}(T_m)$ or in
 $\mathcal{X}(T_m)$. Putting $T=T_m$ in the proof of the case $m=1$,
 one gets the desired result.
\end{proof}
\begin{example}
Let $A_{0}=KQ$ be the algebra given by the quiver $Q: 1 \rightarrow 2
\rightarrow 3 \rightarrow 4$ and let $T_{0}$ be the tilting module
${\begin{smallmatrix}1\\2\\3\\4 \end{smallmatrix} \oplus \begin{smallmatrix}1\end{smallmatrix} \oplus \begin{smallmatrix}1\\2\end{smallmatrix} \oplus\begin{smallmatrix}4\end{smallmatrix}}$ in $\mathrm{mod} A_0$. Then
 \begin{enumerate}[\rm(1)]
\item $A_{1}=\mathrm{End}_{A_{0}}T_{0}$ is given by the quiver $Q':
1\xrightarrow {\alpha_{1}} 2 \xrightarrow {\alpha_{2}} 3 \xrightarrow {\alpha_{3}} 4$ with
 the relation $\alpha_{2}\alpha_{3}=0$ and $\mathrm{gl.dim}A_{1}=2$.
\item $T_{1}={\begin{smallmatrix}1\\2\\3 \end{smallmatrix} \oplus \begin{smallmatrix}2\end{smallmatrix} \oplus \begin{smallmatrix}3\\4\end{smallmatrix} \oplus\begin{smallmatrix}4\end{smallmatrix}}$ in $\mathrm{mod}A_{1}$ is a classical tilting module and $A_{2}=\mathrm{End}_{A_{1}}T_{1}$ is given by the quiver $Q'': 1\xrightarrow {\beta_{1}} 2 \xrightarrow {\beta_{2}} 3 \xrightarrow {\beta_{3}} 4$ with relations $\beta_{1}\beta_{2}=0$ and $\beta_{2}\beta_{3}=0$.
\item $\mathrm{gl.dim}A_{2}=3$ implies that $A_{2}$ is iterated tilted but not tilted.
\item The Auslander-Reiten quiver of $A_{2}$ is as
follows:
$$
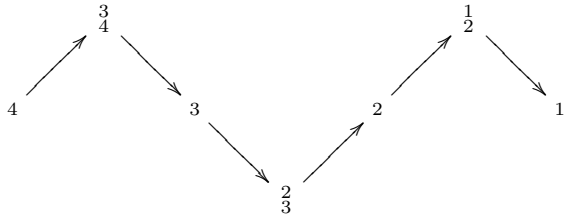
\begin{xy}
 (-36,12)*+{\begin{smallmatrix}4\end{smallmatrix}}="1",
 (-24,24)*+{\begin{smallmatrix}3\\4\end{smallmatrix}}="2",
 (-12,12)*+{\begin{smallmatrix}3\end{smallmatrix}}="3",
(0,0)*+{\begin{smallmatrix} 2\\3\end{smallmatrix}}="4",
(12,12)*+{\begin{smallmatrix} 2\end{smallmatrix}}="5",
(24,24)*+{ \begin{smallmatrix}1\\2\end{smallmatrix}}="6",
(36,12)*+{\begin{smallmatrix} 1\end{smallmatrix}}="7",
\ar"1";"2", \ar"2";"3",  \ar"3";"4",
\ar"4";"5", \ar"5";"6", \ar"6";"7",
\end{xy}$$
\end{enumerate}
One can show that every indecomposable module in $\mathrm{mod}A_{2}$ is $\tau$-rigid.
\end{example}

\section{$\tau$-tilting modules and homological dimension}
In this section, we give the relationship between $\tau$-tilting
modules and homological dimension, which is an analog of that
of classical tilting modules (see \cite[Lemma 4.1]{AsSS} for details).

 For an $A$-module $M$, denote by $M^{\bot_0}$ (resp. ${^{\bot_0}M}$) the subcategory consisting
 of $N$ such that $\Hom_{A}(M,N)=0$ (resp.  $\Hom_{A}(N,M)=0$). Firstly, we introduce the following lemma known as Wakamastu's Lemma.

\begin{lemma}
 \begin{enumerate}[\rm(1)]
\item Let $\theta$: $0 \rightarrow Y \rightarrow T' \xrightarrow {g} X$ be an exact sequence in $\mathrm{mod} A$,
where $T$ is $\tau$-rigid, and $g$: $T'\rightarrow X$ is a right
$\mathrm{add}T$-approximation. Then we have $Y\in {^{\bot_{0}}(\tau
T)}$.
\item Let $\vartheta$: $ Y \xrightarrow {f} U \rightarrow Z \rightarrow 0$ be an exact sequence in $\mathrm{mod} A$,
 where $T$ is $\tau$-rigid, $U\in \mathrm{add}\tau T$, and $f$: $Y \rightarrow U$ is a left $(\mathrm{add}\tau T)$-approximation.
  Then we have $Z\in {T^{\bot_{0}}}$.

\end{enumerate}
\end{lemma}

\begin{proof}
(1) is given by Adachi, Iyama and Reiten in \cite{AIR}. We only
prove (2).

Replacing $Y$ by $\mathrm{Ker}f$, we can assume that $f$ is an
injective. We apply $\mathrm{Hom}_{A}(T,-)$ to $\vartheta$ and get
the exact sequence
$$0= \mathrm{Hom}_{A}(T, U)\rightarrow  \mathrm{Hom}_{A}(T, Z)  \rightarrow \mathrm{Ext}_{A}^{1}(T, Y) \xrightarrow {\mathrm{Ext}_{A}^{1}(T, f)} \mathrm{Ext}_{A}^{1}(T, U)$$
where we have $\mathrm{Hom}_{A}(T, U)=0$ because $U\in
\mathrm{add}\tau T$. Since $f$: $Y\rightarrow U$ is a left
$(\mathrm{add}\tau T)$-approximation, the induced map $(f,\tau T)$:
$\mathrm{Hom}_{A}(U,\tau T)\rightarrow \mathrm{Hom}_{A}(Y,\tau T)$
is surjective. Then the induced map
$\overline{\mathrm{Hom}}_{A}(U,\tau T)\rightarrow
\overline{\mathrm{Hom}}_{A}(Y,\tau T)$ of the maps modulo injectives
is surjective. By the Auslander-Reiten duality, the map
$\mathrm{Ext}_{A}^{1}(T,f): \mathrm{Ext}_{A}^{1}(T,Y)\rightarrow \mathrm{Ext}_{A}^{1}(T,U)$
is injective. It follows that $\mathrm{Hom}_{A}(T, Z)=0$.
\end{proof}

Dually, one can show Wakamastu's Lemma in terms of
$\tau^{-1}$-rigid modules.

Recall from \cite[ChapterVI, Lemma 4.1]{AsSS}, for an algebra $A$,
$T$ a classical tilting module in $\mathrm{mod}A$ and $B=\mathrm{End_A} T$, if $M\in
\mathrm{Fac}T$, then $\pd_{B}\Hom_{A}(T,M)\leq
 \mathrm{pd}_{A}M$ holds. We prove an analog result in terms of
 $\tau$-tilting modules as follows.

\begin{theorem}\label{3.2} Let $A$ be an algebra, $T$ be a $\tau$-tilting module in $\mathrm{mod} A$
 and $B=\mathrm{End}_{A} T$. For any $M\in \mathrm{Fac}T$, we have
 \begin{enumerate}[\rm(1)]
\item   If $\pd_{A} M\leq 1$ holds, then $\mathrm{pd}_{B}\mathrm{Hom}_{A}(T,M)\leq \mathrm{pd}_{A}M$
holds.
\item   If $\Ext_{A}^{i}(T,M\oplus T)=0$ holds
for any $i\geq 1$, then $\mathrm{pd}_{B}\mathrm{Hom}_{A}(T, M)\leq
\mathrm{pd}_{A}M$ holds.

\end{enumerate}
\end{theorem}

\begin{proof}
(1) If $\pd_{A}M=0$, then $M\in \mathrm{Fac}T$ implies $M\in
\mathrm{add}T$. One gets $\mathrm{Hom}_{A}(T,M)$ is a projective
module in $\mathrm{mod}B$ since $\mathrm{Hom}_{A}(T,-)$ induces an equivalence between $\mathrm{add}T$ and $\mathrm{add}B$.

Now, assume $\pd_{A}M=1$. Since $M\in \mathrm{Fac}T$, by Lemma 4.1
we get a short exact sequence

$$0\rightarrow L\rightarrow T_{0} \rightarrow M \rightarrow0\ \ \ \ \ \ \ (4.1)$$
with $L \in {^{\bot_{0}}(\tau T)}=\mathrm{Fac}T$ .

Recall that $L\in \mathcal{C} \subseteq \mathrm{mod}A $ is Ext-projective if $\Ext_{A}^1(L,\mathcal{C})=0$. In the following we show $L\in \add T$, that is, $L$ is
Ext-projective in $\mathrm{Fac}T$.

For any $N\in\mathrm{Fac} T$, applying
the functor $\Hom(-,N)$ to the exact sequence (4.1), we get a long
exact sequence $\Ext_{A}^1(M,N)\rightarrow
\Ext_{A}^1(T_0,N)\rightarrow \Ext_{A}^{1}(L,N)\rightarrow
\Ext_{A}^2(M,N)$. Hence $\Ext_A^1(L,N)=0$ holds because of $\pd_A
M=1$ and $N\in \mathrm{Fac} T$. We are done.

 Applying the functor $\Hom_{A}(T,-)$ to the sequence (4.1) again,
 we get the assertion since $\Hom(T,-)$ is an equivalence between
 $\mathrm{add}T$ and $\mathrm{add}B$.

(2) If $\pd_{A}M=\infty$, then there is nothing to show.

Now we can assume that $\pd_{A}M=t<\infty$.  Since $M\in
\mathrm{Fac}T$, by Lemma 4.1 we get a short exact sequence
$0\rightarrow L\rightarrow T_{0} \rightarrow M \rightarrow0$ with $L
\in {^{\bot_{0}}(\tau T)}=\mathrm{Fac}T$, so $\Ext_{A}^{1}(T,L)=0$. Applying the functor
$\Hom_A(T,-)$ to the sequence (4.1), one gets
$\Ext_{A}^{i+1}(T,L)\simeq \Ext_{A}^{i}(T,M)=0$ for any $i\geq 1$ by
assumption, and hence $\Ext_{A}^{i}(T,L)=0$ for any $i\geq 1$. Continuing the similar process, we get the following
long exact sequence
$$\cdots \rightarrow T_{n} \xrightarrow {f_{n}} T_{n-1}\rightarrow \cdots \rightarrow T_{1} \xrightarrow {f_{1}} T_{0}
\xrightarrow {f_{0}}M \rightarrow 0\ \ \ \ (4.2)$$ with $T_{i}\in
\mathrm{add}T$ and $L_{i+1}=\Ker f_{i}\in {^{\bot_{0}}(\tau
T)}=\mathrm{Fac}T$ for $i\geq0$ and $\Ext_{A}^{j}(T,L_{i+1})=0$ for
$j\geq 1$ and $i\geq 0$.

Next we show that the exact sequence $0\rightarrow
L_{t+1}\rightarrow T_{t}\rightarrow L_{t}\rightarrow 0$ splits.

Since $\pd_{A}M= t<\infty$, then $\Ext_{A}^{t+1}(M,L_{t+1})=0$. On
the other hand, applying the functor $\Hom_{A}(-,L_{t+1})$ to the
sequence (4.2), one gets $0=\Ext_{A}^{t+1}(M,L_{t+1})\simeq
\Ext_{A}^{t}(L_1,L_{t+1})\simeq \cdots\Ext_{A}^{1}(L_t,L_{t+1})$
since $\Ext_{A}^{i}(T,M)=0$ and $L_i\in \mathrm{Fac} T$ hold for any
$i\geq1$. Hence we have a long exact sequence.

$$0\rightarrow T_{t} \xrightarrow {f_{t}} T_{t-1}\rightarrow \cdots
\rightarrow T_{1} \xrightarrow {f_{1}} T_{0} \xrightarrow {f_{0}}M
\rightarrow 0\ \ \ \ (4.3)$$

Applying the functor $\Hom_{A}(T,-)$ to the exact sequence (4.3), we
have
$$0\rightarrow \mathrm{Hom}_{A}(T,T_{t}) \rightarrow \mathrm{Hom}_{A}(T,T_{t-1})
\rightarrow \cdots \rightarrow \mathrm{Hom}_{A}(T,T_{1}) \rightarrow
\mathrm{Hom}_{A}(T,T_{0})
 \rightarrow \mathrm{Hom}_{A}(T,M) \rightarrow 0$$
and hence $\mathrm{pd}_{B}\mathrm{Hom}_{A}(T,M)\leq
t=\mathrm{pd}_{A}M$.
\end{proof}

For a module $T$ in $\mathrm{mod}A$, we denote by
$\mathrm{Sub}T=\{N|N\rightarrowtail T^n \ for\  some\  integer\
n\}$. Then we have the following on the injective dimensions.

\begin{theorem}\label{3.a} Let $A$ be an algebra, $T$ be a $\tau$-tilting module in $\mathrm{mod}A$
 and $C=\mathrm{End}_A \tau T ^{\rm op}$. For any $N\in \mathrm{Sub}\tau T$, we have
 \begin{enumerate}[\rm(1)]
\item   If $\id_{A} N\leq 1$ holds, then $\mathrm{pd}_{C}\mathrm{Hom}_{A}(N,\tau T)\leq \mathrm{id}_{A}N$
holds.
\item   If $\Ext_{A}^{i}(\tau T\oplus N,\tau T)=0$ holds
for any $i\geq 1$, then $\mathrm{pd}_{C}\mathrm{Hom}_{A}(N,\tau
T)\leq \mathrm{id}_{A}N$ holds.

\end{enumerate}
\end{theorem}

\begin{proof} Throughout the proof, we denote by $U=\tau T$.

(1) If $\id_{A} N=0$, then $N\in \mathrm{Sub}U$ implies $N\in
\mathrm{add}U$. One gets $\mathrm{Hom}_{A}(N,U)$ is a projective $C$-module
since $\mathrm{Hom}_{A}(-,U)$ induces a duality between $\mathrm{add}U$ and $\mathrm{add}C$.

Assume $\id_{A} N=1$. Since $N\in \mathrm{Sub}U$, by Lemma 4.1
we get a short exact sequence

$$0\rightarrow N\rightarrow U_{0} \rightarrow L \rightarrow0\ \ \ \ \ (4.4)$$

\noindent where $L \in T^{\bot_{0}}=\mathrm{Sub}U$. In the following
we show $L\in \mathrm{add} U$, that is, $\Ext_{A}^1(N',L)=0$ holds for any
$N'\in \mathrm{Sub} U$. Applying the functor $\Hom_A(N',-)$ to the
exact sequence (4.4), one gets the exact sequence
$\Ext_A^1(N',U)\rightarrow \Ext_A^1(N',L)\rightarrow\Ext_A^2(N',N)$.
The assertion follows from the facts $U$ is Ext-injective and $\id_A
N=1$.

(2) If $\id_A N=\infty$, then there is nothing to show. So we can
assume that $\id_A N=s<\infty$.

By Lemma 4.1 we get the following exact sequence
$$0\rightarrow N\xrightarrow{f_0}U_0\xrightarrow{f_1}U_1\cdots\xrightarrow{f_{s}}U_s\rightarrow\cdots\ \ \ \ \ (4.5)$$
with $f_i$ the minimal left $\mathrm{add} U$-approximation. Denote by
$L_i=\Im f_i$, then one gets $\Ext_{A}^k(L_i,U)=0$ for any $k\geq
1$ and $i\geq 0$.

In the following we show the exact sequence $0\rightarrow
L_s\rightarrow U_s\rightarrow L_{s+1}\rightarrow0$ splits. Applying
the functor $\Hom_A(L_{s+1},-)$ to the exact sequence (4.5), one
gets $0=\Ext_A^{s+1}(L_{s+1},N)\simeq
\Ext_A^{s}(L_{s+1},L_1)\simeq\cdots \simeq\Ext_A^{1}(L_{s+1},L_{s})
$ since $\id_A N\leq n$. Hence we have the following exact sequence

$$0\rightarrow N\xrightarrow{f_0}U_0\xrightarrow {f_1}\cdots\xrightarrow{f_{s}} U_s\rightarrow 0$$
Applying the functor $\Hom_A(-,U)$, one gets the assertion since
$\Ext_A^k(L_i,U)=0$ holds for any $k,i\geq 1$.
\end{proof}

At the end of this section, we give an example to show our main
results.
\begin{example}\label{3.5}
Let $A$ be the algebra given by the quiver $Q: \xymatrix{1\ar@<.2em>[r]^{\alpha_{1}}&2\ar@<.2em>[l]^{\beta_{2}}\ar@<.2em>[r]^{\alpha_{2}}&3\ar@<.2em>[l]^{\beta_{1}}}$ with relations $\alpha_{1}\beta_{2}=0$ and $\alpha_{2}\beta_{1}=\beta_{2}\alpha_{1}$. Then
 \begin{enumerate}[\rm(1)]
\item $A$ is an Auslander algebra and $T= {\begin{smallmatrix} 1\\ &2\\
&&3\end{smallmatrix}\oplus \begin{smallmatrix} &2\\ 1&&3\\
&2\\&&3 \end{smallmatrix} \oplus \begin{smallmatrix}  &2\\
1\\ \end{smallmatrix}}$ is a $\tau$-tilting module in $\mathrm{mod} A$.
\item   $B=\mathrm{End}_{A}T$ is given by the quiver $Q': \xymatrix{3\ar@<.2em>[r]^{\gamma_{3}}&2\ar@<.2em>[r]^{\gamma_{2}}&1\ar@<.2em>[l]^{\gamma_{1}}}$ with the relation $\gamma_{1}\gamma_{2}=0$ and $\mathrm{gl.dim}B=2$.
\item One can show $M={\begin{smallmatrix} 2\\ &3\end{smallmatrix}} \in \mathrm{Fac}T$ with $\mathrm{pd}_{A}M=1$, $\mathrm{Hom}_{A}(T, M)=S(2)$ in $\mathrm{mod}B$, and $\pd_{B} \mathrm{Hom}_{A}(T, M)\leq \pd_{A} M$.
\end{enumerate}
\end{example}

\noindent{\bf Acknowledgements} The first author wants to thank
Professor Zhaoyong Huang for supervision and continuous
encouragement. The third author would like to thank Professor Osamu
Iyama and Professor Bin Zhu for hospitality and useful discussion.
All authors would like to thank the referee for useful suggestions to improve this paper.

\end{document}